\documentclass[12pt]{amsart}
\usepackage[foot]{amsaddr}
\usepackage{amssymb}
\usepackage{hyperref}
\usepackage[all]{xy}
\usepackage{graphicx}

\usepackage{enumitem}

\usepackage{float}

\theoremstyle{plain}
\newtheorem{theorem}{Theorem}[section]
\newtheorem{proposition}[theorem]{Proposition}
\newtheorem{corollary}[theorem]{Corollary}
\newtheorem{lemma}[theorem]{Lemma}

\newtheorem{definition}[theorem]{Definition}

\theoremstyle{remark}
\newtheorem{remark}[theorem]{Remark}
\newtheorem{example}[theorem]{Example}

\numberwithin{equation}{section}

\newcommand{\K}{\ensuremath{\Bbbk}}

\newcommand{\C}{\mathcal C}

\newcommand{\rad}{\operatorname{rad}}
\newcommand{\Hom}{\operatorname{Hom}}
\newcommand{\Ker}{\operatorname{Ker}}

\newcommand{\Ext}{\operatorname{Ext}}
\renewcommand{\Im}{\operatorname{Im}}

\newcommand{\Id}{\textsl{Id}}

\renewcommand{\mod}{\operatorname{mod}}
\newcommand{\inc}{\operatorname{inc}}
\newcommand{\HH}{\ensuremath{\mathsf{HH}}}

\begin{document}

\title{The Ext-algebra for infinitesimal deformations}

\author[M. J. Redondo]{Mar\'\i a Julia Redondo $^{1, 2}$}

\author[L. Rom\'an]{Lucrecia Rom\'an $^1$}

\author[F. Rossi Bertone]{Fiorela Rossi Bertone $^1$}

\address{$^1$ Instituto de Matem\'atica (INMABB), Departamento de Matem\'atica, Universidad Nacional del Sur (UNS)-CONICET, Bah\'\i a Blanca, Argentina}
\address{$^2$ Guangdong Technion Israel Institute of Technology, Shantou, Guangdong Province, China}

\email{mredondo@uns.edu.ar, lroman@uns.edu.ar, fiorela.rossi@uns.edu.ar}



\thanks{The first and the third authors are research members of CONICET (Argentina). The authors  have been  supported  by  the  project  PICT-2019-03360.}


\begin{abstract}
Let $f$ be a Hochschild $2$-cocycle and let $A_f$ be an infinitesimal deformation of an associative finite dimensional algebra $A$ over an algebraically closed field $\K$.  We investigate the algebra structure of the Ext-algebra of $A_f$ and, under some conditions on $f$, we describe it in terms of the Ext-algebra of $A$.  We achieve this description by getting an explicit construction of minimal projective resolutions in $\mod A_f$.
\end{abstract}

\maketitle
\section*{Introduction}

For any associative $\K$-algebra $A$ and $A$-module $M$, the Ext-algebra of $M$ is the vector space $\Ext^*_A(M,M)=\bigoplus_{i=0}^\infty \Ext_A^i(M,M)$ that, equipped with the Yoneda product, is also an associative algebra. Such structure has been discussed for different classes of algebras, see for instance \cite{B2, G1, GK1, GK2, GMMZ, GSST, GZ}. 

On the one hand, the Ext-algebra can be defined in terms of equivalence classes of long exact sequences and, in this case, the multiplication corresponds to splicing exact sequences. 
On the other hand, since it is the cohomology algebra of the endomorphism
dg-algebra $\operatorname{RHom}_A(P,P)$ for a projective resolution $P$ of $M$, the multiplication can be defined in terms of composition of morphisms, see \cite[Theorem 9.1]{HS}  and \cite[\S 2.6]{B1}.
The last point of view shows that it has more structure than merely a graded associative multiplication: it is equipped with an $A_{\infty}$-algebra structure, see \cite{K1}. 
The well-known Koszul duality says  that there is a one to one correspondence between the algebras $A$ and $\Ext_A(S,S)$ for $S \simeq A/\rad A$ when $A$ is a Koszul algebra, see {\cite{BGS,  GRS}}.  This duality does not hold for  any algebra. However, under some assumptions on $A$, the $A_{\infty}$-algebra structure of the Ext-algebra carries enough information to recover the original algebra $A$, see \cite{LPWZ}  and \cite{K2}.

In the 1960s Gerstenhaber introduced the algebraic deformation theory of associative algebras. He studied and described properties of deformations and their relations with Hochschild cohomology groups.

In this work we restrict our attention to infinitesimal deformations. We consider the ring  of dual numbers $\K[t]/(t^2)$ for a field 
$\K$. Then, an infinitesimal deformation of an associative $\K$-algebra $A$ is an associative structure of $\K[t]/(t^2)$-algebra on $A[t]/(t^2)$ such that, when $t=0$, we recover the multiplication in $A$.
Gerstenhaber showed in \cite{G2} that, given an associative algebra $A$, its infinitesimal deformations are parametrized by the second Hochschild cohomology group $\HH^2(A)$ of $A$ with coefficients in itself.

In \cite{RRRV}, for any associative finite dimensional algebra $A$ over an algebraically closed field $\K$ and any Hochschild $2$-cocycle $f$, a description of the finite dimensional modules over  the infinitesimal deformation $A_f$ of $A$ as tuples $(M_0,M_1,T_M,f_M)$ was given.
This work was intended as an attempt to start studying the module category as well as its relation with the category of modules over the original algebra $A$. The aim of this paper is to describe the algebra structure of 
 $\Ext_{A_f}(S, S)$ in terms of $\Ext_{A}(S, S)$. The main theorems  (Theorem \ref{thm yoneda} and  \ref{thm yoneda b}) provide such description under additional assumptions on $f$.
For that, we first construct minimal projective resolutions of simple $A_f$-modules in terms of the corresponding minimal projective resolutions in $\mod A$. It is to be expected that
the $A_\infty$-structure of $\Ext_{A_f}(S, S)$ should be described in terms of $\Ext_{A}(S, S)$.

The paper is organized as follows. In the first section we introduce the basic concepts and notation. The second one is devoted  to the description of the structure of projective $A_f$-modules. In Section 3 we focus our attention on describing minimal projective resolutions of $A_f$-modules from minimal projective resolutions of $A$-modules.  Section 4 provides a detailed description of the liftings needed to calculate Yoneda products. In the last section we apply the previous results in order to get a description of the structure of the algebra $\Ext^*_{A_f}(S,S)$ in terms of the multiplication in $\Ext^*_{A}(S,S)$. 

The authors thank the referee for many helpful remarks and a detailed reading of the manuscript.

\section{Preliminaries}

\subsection{Deformations of associative algebras }\label{subsect:DeforAlg}

Consider the truncated polynomial ring $\K[t]/(t^{2})$. An \emph{infinitesimal deformation} of an associative $\K$-algebra $A$ is an associative $\K[t]/(t^{2})$-algebra structure on $A[t]/(t^{2})$ such that modulo the ideal generated by $t$, the multiplication corresponds to that on $A$. More precisely:
\begin{definition}\label{def:inf def}
	Let  
	$f \in \Hom_\K(A \otimes_\K A, A)$.
	Let $A_f \colon = A[t]/(t^{2})\simeq A\oplus A$ be the algebra with multiplication $\mu$  given by
	\begin{align*}
		\mu(a_0+a_1t, b_0+b_1t) = a_0 b_0 + (a_0 b_1 + a_1 b_0 + f(a_0, b_0))t.
	\end{align*}
	If this product is associative we say that $A_f$ is an \emph{infinitesimal deformation of $A$}.
\end{definition}

 A classical result tells us that 
	the product $\mu$ defines an infinitesimal deformation if and only if $f$ is a Hochschild $2$-cocycle. \\

It is well known that any finite dimensional algebra $A$ over an algebraically closed field $\K$ is Morita equivalent to a quotient of a path algebra, that is, $A$ is Morita equivalent to $\K Q/I$.
A presentation by quiver and relations for $A_f$ and a proof of the good behaviour of deformations under Morita equivalence can be found in \cite{RRRV}. 

\subsection{$A_f$-modules}

In \cite{RRRV} it has been proved that the category  $\mod A_f$ of finite dimensional modules over $A_f$ is equivalent to the category $\C_f$  whose objects are tuples $(M_0,M_1,T_M,f_M)$ with
$M_0, M_1 \in \mod A$, $T_M \in \Hom_A(M_0, M_1)$ a monomorphism and $f_M \in \Hom_\K (A \otimes_\K M_0, M_1)$, satisfying the following condition
\begin{align} \label{eqn:1a}
	af_M(b \otimes m_0) - f_M (ab \otimes m_0) + f_M(a \otimes bm_0) - f(a \otimes b) T_M(m_0) =0
\end{align}
for  $a,b \in A, m_0 \in M_0$. More precisely, the $A_f$-module associated to $(M_0,M_1,T_M,f_M)$ is the $\K$-vector space 
$M_0 \oplus M_1$ with $A_f$-module structure given by 
\begin{align*}
		(a+bt)(m_0, m_1) = (am_0 , am_1 + b T_M(m_0) + f_M(a \otimes m_0))
	\end{align*}
	for $a+bt \in A_f, m_0 \in M_0$ and $m_1 \in M_1$.

The morphisms between objects $(M_0,M_1,T_M,f_M)$ and $(N_0,N_1,T_N,f_N)$ are triples of morphisms $(u_0, u_1, u_2)$ where $u_0 \in \Hom_A(M_0, N_0)$, $u_1 \in \Hom_\K (M_0, N_1)$ and $u_2 \in \Hom_A (M_1, N_1)$ are such that the diagram
\begin{align*}
	\xymatrix{
		M_0 \ar[d]_{T_M} \ar[r]^{u_0} & N_0  \ar[d]^{T_N}\\
		M_1  \ar[r]^{u_2}  & N_1 }
\end{align*}
commutes and, for any $a \in A, m_0 \in M_0$, 
\begin{align}\label{eq:u_1}
	u_1(a m_0) & = a u_1 (m_0) - u_2 (f_M(a \otimes m_0)) + f_N(a \otimes u_0(m_0))
\end{align}
and they correspond to the $A_f$-morphisms $u: M_0 \oplus M_1 \to  N_0 \oplus N_1$ given by 
\begin{align*}
		u(m_0,m_1) = (u_0(m_0) , u_1(m_0) + u_2(m_1))
	\end{align*}
	for $m_0 \in M_0$ and $m_1 \in M_1$. 
The composition is given by
$$(v_0, v_1, v_2) (u_0, u_1, u_2)= (v_0 u_0, v_2 u_1 + v_1 u_0, v_2 u_2)$$
and the identity morphism is $(\Id, 0, \Id)$.

Throughout this article we identify  $\mod A_f$ with $\C_f$. Moreover, there is a full embedding $\mod A \hookrightarrow \mod A_f$ defined by $M \mapsto (0,M,0,0)$, that is, $\mod A$ can be identified with a full subcategory of $\mod A_f$.  By abuse of notation we will denote by $M$ the $A_f$-module $(0,M,0,0)$.

\section{ Projective $A_f$-modules}

From now on, we assume the algebra $A$ is finite dimensional.
Let $A = \oplus_{i=1}^n P_i$ be the decomposition in indecomposable modules.  By classical theory on associative algebras we have that $\{ P_1, \dots , P_n \}$ is a complete set of representatives of isomorphism 
classes of indecomposable projective $A$-modules.  More precisely, $1 = \sum_{i=1}^n e_i$, $P_i= A e_i$ and $\{ e_1, \dots , e_n \}$ is a complete set of primitive orthogonal idempotents. This decomposition allows us to 
describe, up to isomorphism, all the simple modules, that is, $\{ S_1, \dots , S_n \}$ is a complete set of representatives of indecomposable simple $A$-modules, where
$S_i = P_i / \rad P_i$ and $\rad P_i = \rad A . P_i$.

 For any $A$-module $M$, there is an isomorphism of vector spaces 
$$\Hom_A (A e_i, M)  \simeq e_i M$$
given by $h(a e_i) = a h(e_i)$ with $h(e_i) = e_i h(e_i) \in e_i M$. In particular,
morphisms between indecomposable projective modules are uniquely determined by elements in $A$, that is,  
\begin{equation} \label{matriz} \Hom_A (A e_i, Ae_j)  \simeq e_i A e_j.
\end{equation}

\subsection{Projective modules}

 The direct sum of the $A_f$-modules $(M_0,M_1,T_M,f_M)$ and $(N_0,N_1,T_N,f_N)$ is isomorphic to 
\[(M_0\oplus N_0, M_1\oplus N_1, T_{M\oplus N}, f_{M\oplus N})\]
where  $T_{M\oplus N}(m_0+n_0)=T_M(m_0)+T_N(n_0)$ and $f_{M\oplus N}(a\otimes (m_0+n_0))= f_{M}(a\otimes m_0)+f_{N}(a\otimes n_0)$. In particular
\[ A_f = (A,A, \Id, f) \simeq \bigoplus_{i=1}^n (P_i, P_i, \Id, f_{P_i}) \]
where $f_{P_i} \colon A \otimes P_i   \to P_i$ is given by $f_{P_i}(a \otimes be_i) = f(a \otimes b) e_i$. 
Hence $\{ \hat P_1, \dots , \hat P_n \}$ is a complete set of representatives of indecomposable projective $A_f$-modules, with 
$\hat P= (P,P, \Id, f_P)$.  From now on we use this notation for projective $A_f$-modules.

\begin{proposition}\label{u_1}
Let $P=Ae_i, Q=Ae_j$ be  indecomposable projective $A$-modules. Then $(u_0, u_1, u_2) \in \Hom_{A_f}(\hat P, \hat Q)$ if and only if there exist $b, c \in e_iAe_j$ such that
\begin{align*}
u_0(ae_i) & = u_2(ae_i) = a b, \\
 u_1(ae_i) & = f(a \otimes b)e_j + a c
 \end{align*}
 for any $a \in A$.
\end{proposition}

\begin{proof}
Let $(u_0, u_1, u_2) \in \Hom_{A_f}(\hat P, \hat Q)$.
We have seen that any  $u \in  \Hom_A (P,Q)$ is given by $u(ae_i)=ab$ for some $b \in e_iAe_j$, and it is clear that $u_0=u_2$. Using that $f$ is a $2$-cocycle we get that
\begin{align*}
u_1(a' (ae_i)) & = a' u_1 (a e_i) - u_2 (f_P(a' \otimes ae_i)) +  f_Q(a' \otimes u_0(ae_i))\\
& = a' u_1 (a e_i) - u_2 (f(a' \otimes a)e_i) +  f_Q(a' \otimes abe_j)\\
& = a' u_1 (ae_i) - f(a' \otimes a)b + f(a' \otimes ab)e_j\\
& = a' u_1(ae_i) - a' f(a \otimes b)e_j + f(a'a \otimes b)e_j,
\end{align*}
so the map $g \colon P \to Q$ given by $g(ae_i)= u_1(ae_i) - f(a \otimes b)e_j$ is a morphism of $A$-modules, that is, $u_1(ae_i) - f(a \otimes b)e_j = ac$ for some $c \in e_iAe_j$. The converse is straightforward.
\end{proof}

 In order to get a better description of $A_f$-morphisms between projective modules, we define $\tilde{f}: M_{r\times s}(A)\otimes M_{s\times t}(A)\to M_{r\times t}(A)$ by 
 \[ \tilde{f}(B\otimes B')_{ij}=\sum_{l=1}^{s} f(b_{il}\otimes b'_{lj}),\]
for any $B=(b_{il})\in M_{r\times s}(A)$ and $B'=(b'_{lj})\in M_{s\times t}(A)$.
A direct computation shows that $\tilde f$ satisfies the $2$-cocycle condition, that is, 
$$B_0 \tilde f (B_1 \otimes B_2) - \tilde f ( B_0 B_1 \otimes B_2) +  \tilde f ( B_0  \otimes B_1 B_2) - \tilde f ( B_0  \otimes B_1) B_2 =0$$
for any matrices $B_0, B_1, B_2$ of appropriate sizes.

 For any projective $A$-module $Q$ we consider its decomposition in indecomposable projective modules $Q = \oplus_{r=1}^m A e_{i_r}$, we denote 
 $[x] = [x_1 \cdots x_m]$ for any $x=\sum x_r \in Q, x_r \in  A e_{i_r}$,  and we set $E \in M_m(A)$ the diagonal matrix defined by
\[(E)_{st} = 
\begin{cases}
e_{i_s} & \mbox{if $s=t$}, \\
0 & \mbox{otherwise.}
\end{cases}\]

	\begin{proposition}
Let $Q = \oplus_{r=1}^m A e_{i_r} , Q'= \oplus_{r=1}^n A e_{j_r}$ be two projective $A$-modules, $E,E'$ the diagonal matrices associated to them.
Then $(\delta_0,  \delta_1, \delta_2) \in \Hom_{A_f}( \hat Q, \hat Q')$ if and only if there exists
a pair of matrices $B, C \in M_{m \times n}(A)$ such that $EBE'=B$ and $ECE'=C$, and 
\begin{align*}
[\delta_0(x)]&= [\delta_2(x)]= [x] B, \\
[\delta_1(x)] &= \tilde f([x] \otimes B)  E' +[ x ]C
\end{align*}
 for any $x \in Q$.
\end{proposition}

\begin{proof} 
We know that any morphism between projective $A$-modules is given by right multiplication by a matrix with coefficients in $A$. Indeed, using the isomorphism given in \eqref{matriz}, we conclude that 
$[\delta (x)] = [x] B$ for $B=(b_{st}) \in  M_{m \times n}(A)$ with $b_{st} \in e_{i_s} A e_{j_t}$.
Hence 
\[ \Hom_A(Q, Q') \simeq \{ B \in M_{m \times n}(A) : EBE'=B \} .\] 
Now the statement follows by Proposition \ref{u_1}.
\end{proof}

\begin{proposition}\label{cero}  Let $P, M$ be $A$-modules, $P$ projective. The space $\Hom_{A_f}(\hat P, M)$ is isomorphic to $\Hom_A(P,M)$.
\end{proposition}

\begin{proof}  If $(u_0, u_1, u_2) \colon (P,P, \Id, f_P) \to (0, M, 0, 0)$ is an $A_f$-morphism then $u_2=u_0=0$, and  the $\K$-linear map $u_1\colon P \to M$ is a morphism of $A$-modules.
\end{proof}

\subsection{Simple modules}

In \cite{RRRV} we have seen that $\rad A_f= (\rad A, A, \inc, f)$ and hence $\rad \hat P= (\rad P, P, \inc, f_P)$, for any projective $A$-module $P$. Set $S= P/\rad P$.
The exactness of the sequence 
\begin{equation}\label{exacta}
0 \longrightarrow (\rad P, P,  \inc, f_P) \xrightarrow {(\inc, \inc, \Id) }
 (P,P,\Id, f_P) \xrightarrow{(0,  \pi,  0)}  (0,S, 0,0) \to  0 
 \end{equation}
implies that 
$\{ S_1, \dots , S_n \}$ is a complete set of representatives of indecomposable simple $A_f$-modules, 
  where $\pi: P \to S$ denotes the canonical epimorphism.

\begin{remark}
From the previous short exact sequence and the description of the projective $A_f$-modules we can conclude that $A_f$ always has infinite global dimension, since $\hat P$ has even dimension as a vector space, and $S$ and the consecutive kernels have odd dimension.
\end{remark}

 \section{Projective resolutions}
 
The aim of this section is to construct projective resolutions for any $A_f$-module.  Since for any $A_f$-module $(M_0,M_1,T_M,f_M)$ we have a short exact sequence
\[ 0 \to (0, M_1, 0, 0 ) \to (M_0,M_1,T_M,f_M) \to (0, M_0, 0, 0) \to 0,\]
it is enough  to construct projective resolutions for any $A_f$-module $M=(0,M,0,0)$, see  \cite[Horseshoe Lemma 2.2.8] {W}.

\subsection{Examples}
We illustrate our desired results with some specific examples that show that  we can construct minimal projective resolutions for any $A_f$-module $M=(0,M,0,0)$ in terms of minimal projective resolutions of the $A$-module $M$.

By abuse of notation we will denote by $S_i$ the simple module associated to the vertex $i$ in $A$ and $A_f$, that is $S_i=(0,S_i,0,0) $. Also, we denote by $P_i$ and $\hat P_i$ the indecomposable projective modules associated with the vertex $i$ over $A$ and $A_f$ respectively.

\begin{example}\label{example 1}  Let $A=\K Q/I$ with  \[   Q: \xymatrix{ 
		&  2  \ar[ld]_{\alpha_1 } &   \\  1    & & 4 \ar[lu]_{ \alpha_2 } \ar[ld]_{ \alpha_4 }\\
		&  3 \ar[lu]_{\alpha_3}  & } 
		\] 
		and  $I=<\alpha_1\alpha_2>$. The minimal projective resolutions of the simple $A$-modules  are 
		\begin{align*}
		 &0 \to  P_1 \to S_1  \to 0, \\
		 & 0 \to P_1  \to  P_2  \to S_2  \to 0,\\
		 & 0 \to P_1  \to  P_3  \to S_3  \to 0,\\
		  & 0 \to P_1 \to P_2\oplus P_3   \to  P_4  \to S_4 \to 0.
		\end{align*}
	Let $\mathcal B$ be the basis of $A$ given by paths. Let $f \in \Hom_\K(A \otimes A, A)$ be defined on the basis $\mathcal B \otimes \mathcal B$ by 
	\[f(\alpha_1 \otimes \alpha_2)=\alpha_3\alpha_4\] and zero otherwise. A direct computation over the basis $\mathcal B \otimes \mathcal B \otimes \mathcal B$ shows that  $f$ is  a Hochschild $2$-cocycle.
	The infinitesimal deformation $A_f$ (see \cite{RRRV}) is presented by the quiver 
	\[   Q_f: \xymatrix{ 
			&  2 \ar@(ul,ur)^{\gamma_2} \ar[ld]_{\alpha_1 } &   \\  1 \ar@(dl,ul)^{\gamma_1}   & & 4 \ar@(ur,dr)^{\gamma_4} \ar[ld]_{ \alpha_4 } \ar[lu]_{ \alpha_2 }\\ &  3 \ar@(dl,dr)_{\gamma_3} \ar[lu]_{\alpha_3}  & } \]
			 and the ideal $I_f=<\gamma_1^2, \gamma_2^2, \gamma_3^2, \gamma_4^2, \gamma_1\alpha_1-\alpha_1\gamma_2,  \gamma_2\alpha_2-\alpha_2\gamma_4, \gamma_1\alpha_3-\alpha_3\gamma_3, \gamma_3\alpha_4-\alpha_4\gamma_4, \alpha_1\alpha_2-\alpha_3\alpha_4\gamma_4>$. 
			The minimal projective resolutions of the simple $A_f$-modules are 
	\begin{align*}
	&\cdots \to \hat P_1 \to \hat P_1  \to  \hat P_1  \to  S_1  \to 0 ,\\
	& \cdots \to \hat P_2 \oplus \hat P_1 \to \hat P_2 \oplus \hat P_1   \to \hat P_2   {\longrightarrow} S_2 \to 0,\\
		& \cdots \to \hat P_3 \oplus \hat P_1 \to \hat P_3 \oplus \hat P_1   \to \hat P_3   {\longrightarrow} S_3 \to 0,\\
	& \cdots \to \hat P_4 \oplus \hat P_2 \oplus \hat P_3  \oplus \hat P_1 \ \to \hat P_4 \oplus \hat P_2 \oplus \hat P_3   \to  \hat P_4  \to S_4 \to 0.
	\end{align*} 
\end{example}

 \begin{example}\label{example 2} Let $A=\K Q/I$ with  \[  Q:  \xymatrix{
   1 \ar@<6pt>[rr]^{ \alpha_2} & &     2    \ar[ll]^{\alpha_1}  } \]
   and $I$ generated by $\alpha_1\alpha_2\alpha_1, \alpha_2\alpha_1\alpha_2$. The minimal projective resolutions of the simple $A$-modules  are  
		\begin{align*}
		&  \cdots  \to P_1 \to  P_1  \to P_2 \to P_2 \to P_1 \to S_1  \to 0 ,\\
		&  \cdots \to P_2 \to  P_2  \to P_1 \to P_1 \to P_2 \to S_2  \to 0 .
		\end{align*}
Let $\mathcal B$ be the basis of $A$ given by paths. Let $f \in \Hom_\K(A \otimes A, A)$ be the Hochschild $2$-cocycle defined on $\mathcal B \otimes \mathcal B$ by
 \[f(a \otimes b) = \begin{cases}
	 a'\alpha_1 , &\mbox{if $ab=a'\alpha_1\alpha_2\alpha_1 $}, \\
	 a'\alpha_2 , &\mbox{if $ab=a'\alpha_2\alpha_1\alpha_2 $}, \\
	0, &\mbox{otherwise.}
	\end{cases} \]
The infinitesimal deformation $A_f$  is presented by the quiver 
	\[   Q_f: \xymatrix{
	   1 \ar@(dl,ul)^{\gamma_1}\ar@<6pt>[rr]^{ \alpha_2} & &     2   \ar@(ur,dr)^{\gamma_2}  \ar[ll]^{\alpha_1}  }
	\]
			 and the ideal $I_f=<\gamma_1^2, \gamma_2^2,  \gamma_1\alpha_1-\alpha_1\gamma_2, \gamma_2\alpha_2-\alpha_2\gamma_1,  \alpha_1\alpha_2\alpha_1-\alpha_1\gamma_2, \alpha_2\alpha_1\alpha_2-\alpha_2\gamma_1>$.
The minimal projective resolutions of the simple $A_f$-modules are  
\begin{align*}
& \cdots \to  \hat P_1 \oplus \hat P_2 \oplus \hat P_2 \oplus \hat P_1 \to \hat P_1 \oplus \hat P_2 \oplus \hat P_2    \to \hat P_1 \oplus \hat P_2  \to \hat P_1 \to S_1  \to 0,\\ 
&  \cdots \to  \hat P_2 \oplus \hat P_1 \oplus \hat P_1 \oplus \hat P_2 \to \hat P_2 \oplus \hat P_1 \oplus \hat P_1    \to \hat P_2 \oplus \hat P_1  \to \hat P_2 \to S_2  \to 0. 
\end{align*} 
\end{example}

\bigskip

\begin{example} 
Let $A=\K Q/I$ with  
\[ Q: 
\xymatrix{
 1 \ar@/^{9mm}/[rrr]^{\beta_1} \ar@/^{4mm}/[rrr]^{\beta_2}   & & & 2 \ar@/^{9mm}/[lll]_{\alpha_1} \ar@/^{4mm}/[lll]_{\alpha_2}   }
 \]  
 and $I$ generated by $\beta_1\alpha_1, \beta_2\alpha_1, \beta_2\alpha_2, \alpha_2\beta_1$. The minimal projective resolution of the simple $A$-module $S_2$ is 
		\begin{align*}
		 & 0 \to P_2\to P_1\to  P_2\oplus P_2\oplus P_2 \to  P_1\oplus P_1 \to P_2 \to S_2 \to 0.
		\end{align*}
		Let $\mathcal B$ be the basis of $A$ given by paths. Let $f \in \Hom_\K(A \otimes A, A)$ be the Hochschild $2$-cocycle defined on $\mathcal B \otimes \mathcal B$ by	
		\[ 
		f(a \otimes b) = 
		\begin{cases}
	a' e_2 b', & \mbox{if $ab= a' \beta_2\alpha_1 b'$}, \\
	0, & \mbox{otherwise.}
	\end{cases} \]
	The infinitesimal deformation $A_f$  is presented by the quiver 
		\[   Q_f: \xymatrix{
			 1 \ar@(dl,ul)^{\gamma_1}\ar@/^{9mm}/[rrr]^{\beta_1} \ar@/^{4mm}/[rrr]^{\beta_2}   & & & 2 \ar@/^{9mm}/[lll]_{\alpha_1} \ar@/^{4mm}/[lll]_{\alpha_2}
  		}\]
				 and the ideal $I_f=<\gamma_1^2, \beta_2\alpha_1\beta_2\alpha_1,  {\beta_1\alpha_1},
				\beta_2\alpha_2, \alpha_2\beta_1, \beta_2\alpha_1\beta_2-\beta_2\gamma_1,
				   \beta_2\alpha_1\beta_1-\beta_1\gamma_1,
				   \gamma_1\alpha_1-\alpha_1\beta_2\alpha_1,
				   \gamma_1\alpha_2-\alpha_2\beta_2\alpha_1>$.
				 One can check that the minimal projective resolution of the simple $A_f$-module $S_2$ is
	\begin{align*}
			  \dots \to \hat P_1\oplus \hat P_1 \oplus \hat P_2\oplus & \hat P_2 \oplus \hat P_2\oplus \hat P_1 \to \hat P_1\oplus \hat P_1 \oplus \hat P_2\oplus \hat P_2 \oplus \hat P_2\oplus \hat P_1  \\
			 & \to \hat P_1\oplus \hat P_1 \oplus \hat P_2\oplus \hat P_2 \oplus \hat P_2 \to  \hat P_1\oplus \hat P_1 \to \hat P_2 \to S_2 \to 0.
			\end{align*} 

\end{example}

 \begin{example} 
Let $A=\K Q/I$ with  $Q$:  \[   \xymatrix{
		&  2 \ar[ld]_{\alpha_1 } &   \\  1 \ar[rr]_{\alpha_3}  & &   3 \ar[lu]_{ \alpha_2 }} \] and  $I$ generated by $\alpha_1\alpha_2\alpha_3, \alpha_2\alpha_3\alpha_1\alpha_2$. The minimal projective resolution of the simple $A$-module $S_1$ is 
		\begin{align*}
		 & 0 \to  P_1 \to P_2 \to  P_1 \to P_3 \to P_1 \to S_1 \to 0.
		\end{align*}
		Let $\mathcal B$ be the basis of $A$ given by paths. Let $f \in \Hom_\K(A \otimes A, A)$ be the Hochschild $2$-cocycle defined on $\mathcal B \otimes \mathcal B$ by
			\[f(a \otimes b) = \begin{cases}
				a'e_1, &\mbox{if $ab=a'\alpha_1\alpha_2\alpha_3 $}, \\
				a'\alpha_2, &\mbox{if $ab=a'\alpha_2\alpha_3\alpha_1\alpha_2$}, \\
				a' \alpha_3, &\mbox{if $ab=a'\alpha_1\alpha_2\alpha_3\alpha_1$},\\
				0, &\mbox{otherwise.}
			\end{cases} \]

		The infinitesimal deformation $A_f$  is presented by the quiver 
			\[   Q_f: \xymatrix{
					&  2 \ar@(ul,ur)^{\gamma_2} \ar[ld]_{\alpha_1 } &   \\  1 \ar[rr]_{\alpha_3}  & &   3 \ar@(ur,dr)^{\gamma_3} \ar[lu]_{ \alpha_2 }  }
					\]
					 and the ideal $I_f=<\alpha_1\alpha_2\alpha_3\alpha_1\alpha_2\alpha_3, \gamma_2^2, \gamma_3^2, 
					 \alpha_1\gamma_2-\alpha_1\alpha_2\alpha_3\alpha_1, \alpha_2\gamma_3-\alpha_2\alpha_3\alpha_1\alpha_2,
					 \gamma_2\alpha_2-\alpha_2\gamma_3, 
					 \gamma_3\alpha_3-\alpha_3\alpha_1\alpha_2\alpha_3 >$.
					   One can check that the minimal projective resolution of the simple $A_f$-module $S_1$ is	
	\begin{align*}
			 & \dots \to  \hat P_3 \oplus \hat P_2 \oplus \hat P_1  \to  \hat P_3 \oplus \hat P_2 \to  \hat P_3 \oplus \hat P_1 \to \hat P_3 \to \hat P_1 \to S_1 \to 0.
			\end{align*} 
\end{example}

\begin{example} \label{example 3.5}
Let $r \geq 2$ and $A=\K Q/I$ with
\[   Q: 
 \xymatrix{  
		1 \ar@(ur,dr)^{\alpha}   } 
		\] 
		and  $I=<\alpha^r>$. The minimal projective resolution of the simple $A$-module $S_1$ is 
		\begin{align*}
		 & \cdots \to P_1 \to P_1  \to  P_1 \to  S_1  \to 0.
		\end{align*}
		Let $\mathcal B=\{ \alpha^i\}_{0 \leq i <r}$.  
Let $f \in \Hom_\K(A \otimes A, A)$ be the Hochschild $2$-cocycle defined on the basis $\mathcal B \otimes \mathcal B$ by
		\[
		f(\alpha^i \otimes \alpha^j) = 
		\begin{cases}
	\alpha^{i+j-r}, & \mbox{if $i+j \geq r$}, \\
	0, & \mbox{otherwise.}
	\end{cases}
	 \]
	The infinitesimal deformation $A_f$ is presented by the quiver $Q_f=Q$ and the ideal $I_f=<\alpha^{2r}>$. One can check that the minimal projective resolution of the simple $A_f$-module $S_1$ is
	\begin{align*}
	& \cdots \to \hat P_1 \to \hat P_1  \to  \hat P_1 \to  S_1  \to 0 .
	\end{align*} 
Hence, in this particular case, the Ext-algebras associated to $A$ and $A_f$ are isomorphic as vector spaces since 
\[ \Ext^n_{A_f}(S,S) \simeq  \Hom_{A_f}(\hat P_1, S) = \Hom_A(P_1, S) \simeq     \Ext^n_{A}(S,S). \]
Moreover, if $r \geq 3$, they are isomorphic as algebras but not as $A_\infty$-algebras, see \cite[Example 6.3]{LPWZ}.
\end{example}

	\begin{example} \label{example 3.6}
Let $r \in \mathbb N$, $A=\K Q/I$ with  $Q$: 
 \[   \xymatrix{
	&  2 \ar[ld]_{\alpha_1 }  &  3 \ar[l]_{ \alpha_2 }  &  \\ 
	 1 \ar[rd]_{ \alpha_n }   &  &  & j \ar@{.>}@/_1pc/[lu]  \\
	&  n \ar[r]_{ \alpha_{n-1} }  & n-1   \ar@{.>}@/_1pc/[ru]  & } \]
	 and  $I$ generated by $(\alpha_i\alpha_{i+1}\cdots\alpha_{i-1})^r$ for $i=1,\dots, n$, where the subscripts are taken modulo $n$.  The minimal projective resolution of the simple $A$-module $S_i$ is 
\begin{align*}
& \dots \to  P_i \to P_{i-1} \to  P_i \to P_{i-1} \to P_i \to S_i \to 0.
\end{align*}
	Let $\mathcal B$ be the basis of $A$ given by paths. Let $f \in \Hom_\K(A \otimes A, A)$ be the Hochschild $2$-cocycle defined on $\mathcal B \otimes \mathcal B$ by
	\[f(a \otimes b) = \begin{cases}
		a' e_i , &\mbox{if $ab=a'(\alpha_i\alpha_{i+1}\cdots\alpha_{i-1} )^r$}, \, i=1,\dots,r \\
				0, &\mbox{otherwise.}
	\end{cases} \]
The infinitesimal deformation $A_f$ is presented by the quiver $Q_f=Q$ and the ideal $I_f$ generated by $(\alpha_i\alpha_{i+1}\cdots\alpha_{i-1})^{2r}$ for $i=1,\dots, n$. One can check that the minimal projective resolution of the simple $A_f$-module $S_1$ is
\begin{align*}
	& \dots \to  \hat P_i \to \hat P_{i-1} \to  \hat P_i \to  \hat P_{i-1} \to \hat P_i \to S_i \to 0.
\end{align*} 
\end{example}

\bigskip

In the previous examples we have constructed a minimal projective resolution for the $A_f$-module $M=(0,M,0,0)$ using the projectives appearing in a minimal projective resolution of the $A$-module $M$. More precisely, 
if 
\[ \cdots  \to Q_1 \stackrel{\delta_1}  {\longrightarrow}  Q_0  \stackrel{\delta_0}  {\longrightarrow} M \to 0 \]
is a minimal projective resolution of $M$, we have constructed a minimal projective resolution of the $A_f$-module $M=(0,M,0,0)$ such that the 
projective appearing in degree $m$ is $\hat Q_m \oplus X$ where $X$ is a direct summand of $\oplus_{i=0}^{m-1}\hat Q_i$. \\

Even though it is complicated to find a general rule to {determine} which direct summand will appear in each degree, we will find sufficient conditions that give an answer 
for the cases $X=\oplus_{i=0}^{m-1}\hat Q_i$ or $X=0$,
that is, a minimal projective resolution of $M= (0,M,0,0)$ is given by
\begin{enumerate}
\item Case (A): \[ \dots \to \bigoplus_{i=0}^m \hat Q_i \stackrel{\hat \delta_m}{\longrightarrow} \bigoplus_{i=0}^{m-1} \hat Q_i \to \dots \to  \hat Q_0 \oplus \hat Q_1 \stackrel{\hat \delta_1}{\longrightarrow} \hat Q_0 \stackrel{\hat \delta_0}{\longrightarrow} M \to 0;\]
\item Case (B): \[ \dots \to \hat Q_m \stackrel{\hat \delta_m}{\longrightarrow}  \hat Q_{m-1} \to \dots \to  \hat Q_1 \stackrel{\hat \delta_1}{\longrightarrow} \hat Q_0 \stackrel{\hat \delta_0}{\longrightarrow} M \to 0.\]

\end{enumerate}

\subsection{Projective resolutions in case (A)}
Let $i>0$. 
If $\delta_i: Q_i\to Q_{i-1}$ are $A$-morphisms and $T_i: Q_i\to Q_{i-1}$ are linear maps, we define  $u_m, v_m \colon \bigoplus_{i=0}^m Q_i \to \bigoplus_{i=0}^{m-1} Q_i$ given by
\[ u_m = \begin{pmatrix}
0 & \delta_1 & 0 & \dots & 0 \\
\vdots & & & &  \vdots \\
0 & \dots & & 0 & \delta_m 
\end{pmatrix} , \qquad v_m = \begin{pmatrix}
\Id & T_1 & 0 & \dots & 0 \\
\vdots & & & &  \vdots \\
0 & \dots & & (-1)^{m-1} \Id  & (-1)^{m-1} T_m
\end{pmatrix}. \]

\begin{theorem} \label{resolucion}
Let $M$ be an $A$-module and let
\[ \cdots  \to Q_1 \stackrel{\delta_1}  {\longrightarrow}  Q_0  \stackrel{\delta_0}  {\longrightarrow} M \to 0 \]
be a minimal projective resolution of $M$. Suppose that  $T_i \colon Q_i \to Q_{i-1}$ are linear maps such that,  for any $i > 0$,
\begin{enumerate}
\item[(i)] \label{item:i} $T_i(ax) - a T_i(x) = (-1)^{i+1} (f_{Q_{i-1}}(a \otimes \delta_i(x)) - \delta_i (f_{Q_{i}}(a \otimes x)))$, \ $\forall a \in A, x \in Q_i$; 
\item[(ii)]   \label{item:ii} $T_i \delta_{i+1} = \delta_i T_{i+1}$. \label{ii}
\end{enumerate} 
If  $\hat \delta_0 = (0, \delta_0, 0)$ and
$\hat \delta_m = (u_m, v_m, u_m)$ for  $m \geq 1$, 
then 
\[ \dots \to \bigoplus_{i=0}^m \hat Q_i \stackrel{\hat \delta_m}{\longrightarrow} \bigoplus_{i=0}^{m-1} \hat Q_i \to \dots \to  \hat Q_0 \oplus \hat Q_1 \stackrel{\hat \delta_1}{\longrightarrow} \hat Q_0 \stackrel{\hat \delta_0}{\longrightarrow} M \to 0\]
is a minimal projective resolution of $M= (0,M,0,0)$.
 \end{theorem}

\begin{proof} It is clear that $\hat \delta_0$ is an epimorphism. A direct computation shows that $\hat \delta_0 \hat \delta_1=0$ since $\delta_0 \delta_1=0$. Let $m >0$. Condition  (i) 
implies that $(u_m, v_m, u_m)$ is a morphism in $\mod A_f$. The vanishing of 
$\delta_{m} \delta_{m+1}$ implies  $u_{m} u_{m+1}=0$, and  $v_{m} u_{m+1}+ u_{m} v_{m+1}=0$ follows from condition  (ii).
 Hence
$(u_{m}, v_{m}, u_{m}) (u_{m+1}, v_{m+1}, u_{m+1})=0$.  \\
We proceed to show that the complex is exact.
If $(x_0, y_0) \in \Ker \hat \delta_0$, then $x_0 = \delta_1(z)$ and  $\hat \delta_1 (y_0 - T_1 (z), z,0,0) =(x_0, y_0)$.
Let $(x_0, \dots, x_m, y_0, \dots, y_m) \in \Ker (u_m, v_m, u_m)$, that is, for all $i$ with $1\leq i\leq m$, we have
\begin{align*}
\delta_i(x_i)=0 \ \mbox{ and } \ 
(-1)^{i-1}(x_{i-1} + T_i(x_i))+ \delta_i(y_i)=0.
\end{align*}
Since $\delta_m(x_m)=0$, there exists $z_{m+1}$ such that $\delta_{m+1}(z_{m+1})= x_m$. Now we define recursively
$z_j = (-1)^j y_j - T_{j+1}(z_{j+1})$ for any $j$ with $0 \leq j \leq m$.  It is a simple matter to check that 
$(u_{m+1}, v_{m+1}, u_{m+1})(z_0, \dots, z_{m+1}, 0, \dots, 0)=(x_0, \dots, x_m, y_0, \dots, y_m)$. \\
 The only point remaining concerns minimality, that is, $\Im \hat \delta_m \subset \bigoplus_{i=0}^{m-1} \rad \hat Q_i$ for any $m >0$. From \eqref{exacta} it suffices to see
 that for any simple module $S = (0, S, 0, 0)$ we have that 
 $$(u_m, v_m, u_m)^*  \colon \Hom_{A_f} ( \bigoplus_{i=0}^{m-1} \hat Q_i , S) \to  \Hom_{A_f} ( \bigoplus_{i=0}^m \hat Q_i , S)$$
  is zero. Let $\hat g \in \Hom_{A_f} ( \bigoplus_{i=0}^{m-1} \hat Q_i , S)$. 
 By Proposition \ref{cero}, $\hat g = (0, g ,0)$ with $g=[g_0 \ \cdots \ g_{m-1}]$,  and $g_i \in \Hom_{A} (Q_i , S)$. 
Therefore,
\[ (u_m, v_m, u_m)^*(\hat g) = \hat g  \ (u_m, v_m, u_m) = (0, g   u_m, 0).  \]
By the minimality of the original resolution,  $\delta_i^* \colon \Hom_A (Q_{i-1}, S) \to \Hom_A(Q_i, S)$ is zero for any $i$. Thus, $g  u_m = [ 0 \ g_0 \delta_1 \ g_1 \delta_2 \ \cdots \ g_{m-1} \delta_{m}]=0$ and the Lemma follows.
 \end{proof}

We say that the $A$-module $M$ satisfies \emph{condition} $(\ast A)$ if it admits a minimal projective resolution 
\[ \cdots  \to Q_1 \stackrel{\delta_1}  {\longrightarrow}  Q_0  \stackrel{\delta_0}  {\longrightarrow} M \to 0 \]
such that $\delta_0 T_{\delta_1}=0$ for  $[T_{\delta_1}(x)] = \tilde{f}([x] \otimes B_1)E_0$,  where $B_i \in M_{n_i \times n_{i-1}}(A)$ is the matrix associated to each morphism $\delta_i : Q_i \to Q_{i-1}$ with $Q_i= \oplus_{r=1}^{n_i} A e_{i_r}$ and  $E_i \in M_{n_i}(A)$ is the diagonal matrix associated  with the projective module $Q_i$.

	\begin{lemma}
		If $M$ satisfies condition $(\ast A)$, then there exist matrices $C_i\in M_{n_i\times n_{i-1}}(A)$ such that $E_i C_i E_{i-1}=C_i$  for any $i \geq 1$ and $E_i \tilde{f}(B_i\otimes B_{i-1})E_{i-2}  = C_iB_{i-1}+B_iC_{i-1}$ for any $i \geq 2$.
\end{lemma}

\begin{proof}
Set $C_1=0$.  Now we proceed by induction. Using that $\tilde f$ satisfies the $2$-cocycle condition and $B_2 B_1 =0$, we get that
 \[ [T_{\delta_1}  \delta_2  (x)] = \tilde f([x] B_2 \otimes B_1)E_0=[x] \tilde f(B_2 \otimes B_1)E_0- \tilde f([x] \otimes B_2)B_1 E_0\]
for any  $x \in Q_2$. Since $\Im T_{\delta_1} \delta_2 \subset \Im T_{\delta_1} \subset \Ker \delta_0 = \Im \delta_1$ we can conclude that there exists $y_r \in Q_1$ such that 
\[  [e_{i_r}] \tilde f(B_2 \otimes B_1)E_0 = [y_r] B_1\]
for any $1 \leq r \leq n_2$. Then $E_2 \tilde{f}(B_2 \otimes B_1) E_0 =  C_2 B_1$ for $C_2=E_2 Y E_1$ where $Y \in M_{n_2\times n_{1}}(A)$ is the matrix whose rows are given by $[y_1], \dots, [y_{n_2}]$.

Let $i >1$ and 
assume that $E_i \tilde{f}(B_i\otimes B_{i-1})E_{i-2} =  C_iB_{i-1}+B_iC_{i-1}$. The vanishing of the composition $\delta_m  \delta_{m+1}$ implies that  $B_{m+1} B_{m} =0$  for all $m>0$.
Then
\[B_{i+1} \tilde{f}(B_i \otimes B_{i-1}) - \tilde{f}(B_{i+1} \otimes B_i) B_{i-1} =0\]
since  $\tilde f$ satisfies the $2$-cocycle condition. Therefore
\begin{align*} 
E_{i+1} \tilde{f}(B_{i+1} \otimes B_i) B_{i-1} & =  B_{i+1} \tilde{f}(B_{i} \otimes B_{i-1})E_{i-2}  = B_{i+1} C_i B_{i-1} .
\end{align*}
This implies that 
\[E_{i+1} \tilde{f}(B_{i+1} \otimes B_i)E_{i-1}  - B_{i+1} C_i \in \Ker \delta_{i-1} = \Im \delta_{i},\] so there exists $C_{i+1}$ such that 
$E_{i+1} \tilde{f}(B_{i+1} \otimes B_i)E_{i-1}  - B_{i+1} C_i  =   C_{i+1} B_i$ and the lemma follows.
\end{proof}

The following results are straightforward. 

\begin{proposition}\label{resolucion star} If $M$ satisfies condition $(\ast A)$
then the linear maps $T_{\delta_i} \colon Q_i \to Q_{i-1}$ defined by 
\begin{equation*} [T_{\delta_i} (x)] =(-1)^{i+1} (\tilde f([x] \otimes B_i )E_{i-1}  - [x] C_i) \label{T}\end{equation*}
with $B_i,C_i$ as before, satisfy conditions (i) and (ii) in Theorem \ref{resolucion}.

\end{proposition}

\begin{theorem}\label{3.4}
For any $A$-module satisfying condition $(\ast A)$, the $A_f$-module $(0,M,0,0)$ admits a minimal projective resolution as described in Theorem \ref{resolucion}. 
\end{theorem}

\subsection{Projective resolutions in case (B)}

We say that the $A$-module $M$ satisfies \emph{condition}  $(\ast B)$  if it admits a minimal projective resolution 
\[ \cdots  \to Q_1 \stackrel{\delta_1}  {\longrightarrow}  Q_0  \stackrel{\delta_0}  {\longrightarrow} M \to 0 \]
such that, for any $i \geq 0$, $Q_i=  \oplus_{r=1}^{n_i}A e_{i_r}$, $B_{i+1} \in M_{n_{i+1} \times n_{i}}(A)$ the matrix associated to each 
morphism $\delta_{i +1}: Q_{i +1}\to Q_{i}$, $E_i \in M_{n_i}(A)$ the diagonal matrix associated to the projective module $Q_i$, $[T_{\delta_{i+1}}(x)] = \tilde{f}([x] \otimes B_{i+1})E_{i}$, then
\begin{enumerate}[label=(\roman*),ref=(\roman*)] 
\item  $\delta_{2i}  T_{\delta_{2i+1}}$ restricted to $\Ker \delta_{2i+1}$ is a monomorphism, and \label{itemone}
\item $Q_{2i} = T_{\delta_{2i+1}}(\Ker \delta_{2i+1}) \oplus \Im \delta_{2i+1}$ as $\K$-vector spaces. \label{itemtwo}
\end{enumerate}

\begin{theorem} \label{resolucion B}
Let $M$ be an $A$-module and let
\[ \cdots  \to Q_1 \stackrel{\delta_1}  {\longrightarrow}  Q_0  \stackrel{\delta_0}  {\longrightarrow} M \to 0 \]
be a minimal projective resolution of $M$ satisfying condition  $(\ast B)$. Then 
\[ \dots \to \hat Q_m \stackrel{\hat \delta_m}{\longrightarrow}  \hat Q_{m-1} \to \dots \to  \hat Q_1 \stackrel{\hat \delta_1}{\longrightarrow} \hat Q_0 \stackrel{\hat \delta_0}{\longrightarrow} M \to 0\]
is a minimal projective resolution of $M= (0,M,0,0)$, where  $\hat \delta_{2i} = (0, \delta_{2i}, 0)$ and
$\hat \delta_{2i+1} = (\delta_{2i+1}, T_{\delta_{2i+1}}, \delta_{2i+1})$ for  $i \geq 0$.
 \end{theorem}
 
  \begin{proof} It is clear that $\hat \delta_0$ is an epimorphism. A direct computation shows that $\hat \delta_{i} \hat \delta_{i+1}=0$ for any $i \geq 0$. 
   We proceed to show that the complex is exact.
  
If $(x, y) \in \Ker \hat \delta_{2i}$, then $x = \delta_{2i+1}(z_0)$. By condition \ref{itemtwo} we have that $$y - T_{\delta_{2i+1}}(z_0)= T_{\delta_{2i+1}}(z_1) +  \delta_{2i+1}(z_2)$$ for $z_1 \in \Ker \delta_{2i+1}$. Then $\hat \delta_{2i+1} (z_0+z_1,z_2) =(x, y)$. 

If $(x, y) \in \Ker \hat \delta_{2i+1}$, then $x \in \Ker \delta_{2i+1}$ and $T_{\delta_{2i+1}}(x) + \delta_{2i+1}(y)=0$. Then  $\delta_{2i} T_{\delta_{2i+1}}(x)=0$ and $x \in \Ker \delta_{2i+1}$. By condition \ref{itemone} we have that $x=0$. So  $y=\delta_{2i+2}(z)$ and $(0,y) =  \hat \delta_{2i+2}(z,0)$.

The minimality follows as in the proof of  Theorem \ref{resolucion}.
 \end{proof}

\subsection { }Taking into account conditions $(\ast A)$ and $(\ast B)$, we get the following theorem.

\begin{theorem} \label{ext}
Let $M,N$  be $A$-modules. If $M$ satisfies condition  $(\ast A)$ then
\[ \Ext_{A_f}^*( M, N) \simeq \Ext_{A}^*( M,  N) \otimes_\K \K [x] \]
and  if $M$ satisfies condition  $(\ast B)$ then
\[ \Ext_{A_f}^*( M, N) \simeq \Ext_{A}^*( M,  N) \]
as graded $\K$-vector spaces.
 \end{theorem}
 
 \begin{proof}
Observe that $\Ext_{A_f}^n( M, N)$ can be computed applying the functor $\Hom_{A_f}^n( - , N)$ to the minimal projective resolution constructed in Theorem \ref{resolucion},
 and the diagram 
\[ \xymatrix{  \Hom_{A_f} (\bigoplus_{k=0}^{m-1} \hat Q_k ,  N) \ar[r]^{\hat \delta_m^*} \ar[d]^{\simeq} & \Hom_{A_f} (\bigoplus_{k=0}^m \hat Q_k,  N) \ar[d]^{\simeq}\\
\bigoplus_{k=0}^{m-1} \Hom_{A}( Q_k,  N)   \ar[r]^{\bigoplus_{k=1}^{m}\delta^*_k}  & \bigoplus_{k=0}^m \Hom_{A}( Q_k,  N)
}
\] 
is commutative. Hence $\Ext_{A_f}^n( M, N) \simeq \bigoplus_{k=0}^n ( \Ext_{A}^k( M,  N) \otimes x^{n-k})$. Analogously, Case (B) follows from  Theorem \ref{resolucion B}.
 \end{proof}

\section{Liftings}

For any $A_f$-module $X$, the $\K$-module
$\Ext_{A_f}^{\ast}(X, X) = \bigoplus_{n \geq 0}\Ext_{A_f}^{n}(X,X)$ can be endowed with  structure of associative
$\K$-algebra using the Yoneda product
$$\circ \colon \Ext_{A_f}^{m}(Y,Z) \times \Ext_{A_f}^{n}(X,Y)\to \Ext_{A_f}^{n+m}(X,Z)$$
given by
$${h} \circ {g} :=  h  {\gamma}_m,$$
 where ${h} \in \Ext_{A_f}^{m}(Y,Z) $, ${g} \in \Ext_{A_f}^{n}(X,Y)$, and $\{ {\gamma}_m\}_{m \geq 0}$ is a lifting of the morphism $g$ defined between given projective resolutions of $X$ and $Y$. \\

Let $M, M'$ be two $A$-modules satisfying simultaneously condition $(\ast A)$,  or $(\ast B)$.  We want to describe an appropriate lifting for any  morphism $\hat g \in \Ext^n_{A_f} (M, M')$ in order to describe the Yoneda product.
For this, we fix
\[ \cdots  \to Q_1 \stackrel{\delta_1}  {\longrightarrow}  Q_0  \stackrel{\delta_0}  {\longrightarrow} M\to 0, \]
\[ \cdots  \to Q'_1 \stackrel{\delta'_1}  {\longrightarrow}  Q'_0  \stackrel{\delta'_0}  {\longrightarrow} M'\to 0 \]
minimal projective resolutions of the $A$-modules $M$ and $M'$.
Recall that $\{ \gamma_m\}_{m \geq 0}$ is a lifting of the morphism $g \in \Hom_A( Q_n , M')$  if 
   \[ \xymatrix{ 
\cdots \ar[r] 
&  Q_{n+m} \ar[r]^{\delta_{n+m}} \ar[d]^{\gamma_{m}} 
 &  Q_{n+m-1} \ar[r]  \ar[d]^{\gamma_{m-1}}  
& \dots  \ar[r] 
& Q_{n+1}  \ar[r]^{\delta_{n+1}} \ar[d]^{\gamma_1} 
& Q_{n} \ar[d]^{ \gamma_0} 
\ar[rd]^{g} 
&  \\
\cdots  \ar[r] 
&  Q'_{m}   \ar[r]^-{\delta'_{m}} 
 &  Q'_{m-1}  \ar[r]
& \dots   \ar[r] 
& Q'_1  \ar[r]^-{\delta'_1}  \ar[r] 
 & Q'_0  \ar[r]^-{\delta'_0}  
& M' \ar[r] 
& 0 }
  \] is a commutative diagram in $A$-$\mod$.

\subsection{Liftings in Case A}

Suppose that the minimal projective resolutions of the $A$-modules $M$ and $M'$
satisfy condition $(\ast A)$, and these projective resolutions are defined as in Proposition \ref{resolucion star}. 
A representative of $\hat g \in  \Ext_{A_f} ^n(M, M')$ is given by a morphism in $\Hom_{A_f}(\bigoplus_{r=0}^n \hat Q_r, M')$, that is,  
$\hat g =(0,g,0)$ with $g= [ g_0 \ \cdots \ g_n]$ for $g_r \in \Hom_A(Q_r, M')$. 
The Lifting Theorem yields the existence of morphisms $\hat \gamma_{m}=  (\gamma_m, \beta_m, \gamma_m)$ such that  
\[ {\tiny  \xymatrix{ 
\cdots \ar[r] &  \bigoplus_{r=0}^{n+m} \hat Q_{r}   \ar[r]^-{\hat \delta_{n+m}} \ar[d]^{\hat \gamma_{m}} & \bigoplus_{r=0}^{n+m-1} \hat Q_{r}    \ar[r] \ar[d]^{\hat \gamma_{m-1}} &
\dots  \ar[r] &  \bigoplus_{r=0}^{n+1} \hat Q_{r}  \ar[r]^{\hat \delta_{n+1}} \ar[d]^{\hat \gamma_{1}} &  \bigoplus_{r=0}^{n} \hat Q_{r} \ar[d]^{\hat \gamma_0} \ar[rd]^{\hat g} &  \\
\cdots  \ar[r] & \bigoplus_{r=0}^m \hat Q'_r   \ar[r]^{\hat \delta'_{m}} &  \bigoplus_{r=0}^{m-1} \hat Q'_r    \ar[r] &  \dots   \ar[r] & \bigoplus_{r=0}^1 \hat Q'_r  \ar[r]^{\hat \delta'_1}  \ar[r] & \hat Q'_0  \ar[r]^{\hat \delta'_0}  & \hat M' \ar[r] & 0  } }\]
is a commutative diagram where
\[ 
\gamma_m =  [ \gamma^{s,m}_t], 
\beta_m =  [ \beta^{s,m}_t] \colon  \bigoplus_{s=0}^{n+m} Q_{s} \to   \bigoplus_{t=0}^m Q'_t \]
with $\gamma^{s,m}_t \in \Hom_A( Q_{s}, Q'_t)$ and $\beta^{s,m}_t \in \Hom_\K( Q_{s}, Q'_t)$.  \\

A straightforward computation shows that these morphisms are characterized by the following properties:

\begin{lemma}\label{a--g}
The property of being morphisms is equivalent to 
\begin{itemize}
\item [a)] for $0 \leq t \leq m$ and $0 \leq s \leq n+m$:
\[  \beta_t^{s,m} (a x) - a  \beta_t^{s,m} (x)  = f_{Q_t}( a \otimes \gamma_t^{s,m}(x)) -  \gamma_t^{s,m}(f_{Q_s}(a \otimes x)), \quad \forall x \in Q_s.\]
\end{itemize}
The commutativity of the diagram is equivalent to conditions
\begin{itemize}
\item [b)] for $0  \leq s \leq n$:
\[ \delta'_0  \gamma^{s,0}_{0}  = g_s;\]
\item [c)] for $1 \leq t \leq m$:
\[ \delta'_t  \gamma_t^{0,m} =0; \]
\item [d)] for  $1 \leq t \leq m$ and $1 \leq s \leq n+m$:
\[ \delta'_t  \gamma^{s,m}_{t} = \gamma^{s-1,m-1}_{t-1}  \delta_{s};\]
\item [e)] for $1 \leq t \leq m$:
\[ (-1)^{t} \gamma_{t-1}^{0,m} + \gamma_{t-1}^{0,m-1}  = \delta'_t  \beta_t^{0,m}  + (-1)^{t-1} T_{\delta'_t }  \gamma_t^{0,m};\]
\item [f)] for $1 \leq t \leq m$ and $1 \leq s < n+m$:
\small{ \[ (-1)^{t} \gamma_{t-1}^{s,m} 
 + (-1)^s \gamma_{t-1}^{s,m-1}  = \delta'_t  \beta_t^{s,m}  - \beta_{t-1}^{s-1, m-1}  \delta_{s} + (-1)^{t-1} T_{\delta'_t}   \gamma_t^{s,m} + (-1)^{s} \gamma_{t-1}^{s-1, m-1}  T_{\delta_{s}}; \]}
\item [g)]  \label{condition} for $1 \leq t \leq m$:
\small{\[  (-1)^{t} \gamma_{t-1}^{n+m,m} =  \delta'_t  \beta_t^{n+m,m} 
- \beta_{t-1}^{n+m-1, m-1}  \delta_{n+m} + (-1)^{t-1} T_{\delta'_t}  \gamma_t^{n+m,m} + (-1)^{n+m} \gamma_{t-1}^{n+m-1, m-1} T_{ \delta_{n+m}}  .\]}
\end{itemize}
\end{lemma}

\bigskip

Starting from the lifting $\{ \hat \gamma_m= (\gamma_m, \beta_m, \gamma_m)\}_{m \geq 0}$, we will construct a more convenient lifting $\{(\varphi_m, \phi_m, \varphi_m)\}_{m \geq 0}$ with  the following shape
$$ \varphi_m = \begin{pmatrix}
	\gamma_0^{0,0} & { \varphi_0^{1,m}} & { \varphi_0^{2,m}} &  \dots & { \varphi_0^{m-1,m}} & \gamma_0^{m,m} &    \dots &  \gamma_0^{n+m,m} \\
	0 & \gamma_1^{1,1} & { \varphi_1^{2,m}} & \dots &  {\varphi_1^{m-1,m} }&  \gamma_1^{m,m}    & \dots &  \gamma_1^{n+m,m} \\ 
	\vdots & &  & &  & \vdots &  &  \vdots \\
	0 &  & \dots &    & \gamma_{m-1}^{m-1,m-1} & \gamma_{m-1}^{m,m} &  \dots   &  \gamma_{m-1}^{n+m,m} \\
	0 &  & \dots &    & 0 & \gamma_{m}^{m,m} &  \dots   &  \gamma_{m}^{n+m,m} 
	\end{pmatrix}  \label{eq:def}
	$$
and
	$$ \phi_m= \begin{pmatrix}
	\beta_0^{0,0} & { \phi_0^{1,m}} & { \phi_0^{2,m}} &  \dots & { \phi_0^{m-1,m}} & \beta_0^{m,m} &    \dots &  \beta_0^{n+m,m} \\
	0 & \beta_1^{1,1} & { \phi_1^{2,m}} & \dots &  { \phi_1^{m-1,m} }&  \beta_1^{m,m}    & \dots &  \beta_1^{n+m,m} \\ 
	\vdots & &  & &  & \vdots &  &  \vdots \\
	0 &  & \dots &    & \beta_{m-1}^{m-1,m-1} & \beta_{m-1}^{m,m} &  \dots   &  \beta_{m-1}^{n+m,m} \\
	0 &  & \dots &    & 0 & \beta_{m}^{m,m} &  \dots   &  \beta_{m}^{n+m,m} 
	\end{pmatrix}.
	$$
		More precisely, we define
\begin{align}\label{eq:varphi}
	\varphi^{s,m}_t = 
	\begin{cases}
		0, & \mbox{if $s < t$}; \\
		\gamma^{s,m}_t, &
		\mbox{if $t \leq m \leq s  \leq n+m$}; \\
		\sum_{i=0}^{r} a^{m-s}_{r,i}\gamma_{t}^{s,s-i}, & \mbox{if $0 \leq s-t= r$, $0 \leq s \leq m$};
	\end{cases} \\
\label{eq:phi} \phi^{s,m}_t = 
\begin{cases}
0, & \mbox{if $s < t$}; \\
\beta^{s,m}_t, & \mbox{if $t \leq m \leq s  \leq n+m$}; \\
\sum_{i=0}^{r} a^{m-s}_{r,i}\beta_{t}^{s,s-i}, & \mbox{if $0 \leq  s-t= r$, $0 \leq s \leq m$};
\end{cases} 
\end{align}	
where $a^{m-s}_{r,i}$ is defined inductively by the formulae
\[ \begin{cases} 
a^0_{r,0}  = 1, & \mbox{for any $r \geq 0$}, \\
a^0_{r,i}  = 0, & \mbox{for any $0< i \leq r $}, \\
a^{k}_{r,i}  = (-1)^r a^{k-1}_{r,i} + a^{k}_{r-1,i}+ (-1)^{r+1} a^{k}_{r-1,i-1}, & \mbox{for any $0 \leq i \leq r, k >0$},
\end{cases} \]
with the convention that $a^{k}_{r,i} =0$ if $r<0$,  $i<0$ or $i>r$. 
In particular, one gets that $a^{m-s}_{0,0}= a^{m-s-1}_{0,0}$, and hence
$a^{m-s}_{0,0}= a^{0}_{0,0}=1$.

\begin{proposition} \label{Proposition 6.1} For any $\hat g  \in \Ext^n_{A_f} (M, M')$, set $\hat \gamma_m$ as before. Then the
	 morphisms $\hat \varphi_{m}=  (\varphi_m, \phi_m, \varphi_m)$ complete the following diagram
	\[ {\tiny  \xymatrix{ 
			\cdots \ar[r] &  \bigoplus_{r=0}^{n+m} \hat Q_{r}   \ar[r]^-{\hat \delta_{n+m}} \ar[d]^{\hat \varphi_{m}} & \bigoplus_{r=0}^{n+m-1} \hat Q_{r}    \ar[r] \ar[d]^{\hat \varphi_{m-1}} &
			\dots  \ar[r] &  \bigoplus_{r=0}^{n+1} \hat Q_{r}  \ar[r]^{\hat \delta_{n+1}} \ar[d]^{\hat \varphi_{1}} &  \bigoplus_{r=0}^{n} \hat Q_{r} \ar[d]^{\hat \varphi_0} \ar[rd]^{\hat g} &  \\
			\cdots  \ar[r] & \bigoplus_{r=0}^m \hat P'_r   \ar[r]^{\hat \delta'_{m}} &  \bigoplus_{r=0}^{m-1} \hat P'_r    \ar[r] &  \dots   \ar[r] & \bigoplus_{r=0}^1 \hat P'_r  \ar[r]^{\hat \delta'_1}  \ar[r] & \hat P'_0  \ar[r]^{\hat \delta'_0}  & \hat M' \ar[r] & 0  } }\]
	in a commutative way.
\end{proposition}
\begin{proof}  	We need to prove that $\varphi^{s,m}_t$ and $\phi^{s,m}_t$ satisfy all the equations in Lemma \ref{a--g}. Equation (a) follows immediately by definition. Since $\varphi^{s,0}_0= \gamma^{s,0}_0$, equation (b) is satisfied. For any $t \geq 1$, $\varphi^{0,m}_t=0$, thus (c) follows.  Equation (d) is satisfied by
$\varphi^{s,m}_{t}$ and $\varphi^{s-1,m-1}_{t-1}$ since $s-t=(s-1)-(t-1)$ and $m-s= (m-1)-(s-1)$, thus they belong to the same case in \eqref{eq:varphi}.

 For the last three equations, to simplify notation set 
\[ \Phi^{s,m}_t = \delta'_t  \phi_t^{s,m}  - \phi_{t-1}^{s-1, m-1}  \delta_{s} + (-1)^{t-1} T_{\delta'_t }  \varphi_t^{s,m} + (-1)^{s} \varphi_{t-1}^{s-1, m-1}  T_{\delta_{s}}.
\]
Since all the terms in this expression belong to the same case in \eqref{eq:varphi} and \eqref{eq:phi}, we have that
\begin{equation} \label{induccion}
 \Phi^{s,m}_t = 
\begin{cases}
0, & \mbox{if $s < t$}; \\
\sum_{i=0}^{r} a^{m-s}_{r,i}\Phi_{t}^{s,s-i}, & \mbox{if $0 \leq s-t= r$, $0 \leq s \leq m$};\\
(-1)^{t} \gamma_{t-1}^{s,m}  + (-1)^s \gamma_{t-1}^{s,m-1}, & \mbox{if $0 \leq s-t$, $m \leq s \leq m+n$}. 
\end{cases} 
\end{equation}
With this notation, equation (e) reads as follows
\[ (-1)^{t} \varphi_{t-1}^{0,m} + \varphi_{t-1}^{0,m-1}  = \Phi_t^{0,m},\]
and it is clearly satisfied for $t>1$ since all the terms vanish, and also for $t=1$ since $\varphi_{0}^{0,m} = \varphi_{0}^{0,m-1}=\varphi_{0}^{0,0}$.
Analogously, equation (f) 
 \[ (-1)^{t} \varphi_{t-1}^{s,m} 
 + (-1)^s \varphi_{t-1}^{s,m-1}  = \Phi_t^{s,m} \]
 is satisfied when $s < t-1 < t$, and also for $s=t-1$ since $\varphi_{s}^{s,m} = \varphi_{s}^{s,m-1}=\varphi_{s}^{s,s}$. For $s-t \geq 0$ and $m \leq s < m+n$ we recover the formula for $\hat \gamma_m$.

Now we prove the equality for $r= s-t \geq 0$ and $1 \leq s < m$.   Let $m -s >0$ and $r \geq 0$. 
By  \eqref{induccion} and the definition of $\Phi^{s, s-i}_{s-r}$ we have that 
 \begin{align*}
    \Phi_{s-r}^{s,m}  & = \sum_{i=0}^r a^{m-s}_{r,i} \Phi^{s, s-i}_{s-r} 
     = \sum_{i=0}^r a^{m-s}_{r,i} ( (-1)^{s-r}    \gamma^{s, s-i}_{s-r-1} + (-1)^s   \gamma^{s, s-i-1}_{s-r-1}) \\
    & = (-1)^{s-r}  \sum_{i=0}^r a^{m-s}_{r,i}    \gamma^{s, s-i}_{s-r-1} +  (-1)^s  \sum_{i=1}^{r+1} a^{m-s}_{r,i-1}     \gamma^{s, s-i}_{s-r-1}, \\    
     & = \sum_{i=0}^{r+1} ( (-1)^{s-r}  a^{m-s}_{r,i}   +  (-1)^s  a^{m-s}_{r,i-1} )    \gamma^{s, s-i}_{s-r-1}. 
   \end{align*}
  Since	$$a^{m-s}_{r+1,i}  = (-1)^{r+1} a^{m-s-1}_{r+1,i} + a^{m-s}_{r,i}+ (-1)^{r} a^{m-s}_{r,i-1},$$
  by \eqref{eq:varphi}   
  equation (f) holds.
 Finally, (g) coincides with the equation for $\hat \gamma_m$.
 \end{proof}

\bigskip

\begin{remark} \label{levantados}
Equation (d) implies that  the morphisms $\varphi_t^{s,m}$ correspond to a lifting in degree $t$ of $\delta'_0   \varphi^{s-t,m-t}_0$ for  $1 \leq t \leq m$ and $t \leq s \leq n+m$. Indeed,
   \[ \xymatrix{ 
\cdots \ar[r] 
&  Q_{s} \ar[r]^{\delta_{s}} \ar[d]_{\varphi^{s, m}_{t}} 
 &  Q_{s-1} \ar[r]  \ar[d]_{\varphi^{s-1, m-1}_{t-1}}  
& \dots  \ar[r] 
& Q_{s-t+1}  \ar[r]^{\delta_{s-t+1}} \ar[d]_{\varphi^{s-t+1,m-t+1}_1} 
& Q_{s-t} \ar[d]_{ \varphi^{s-t,m-t}_0} 
\ar[rd]^{\delta_0'   \varphi^{s-t,m-t}_0 } 
&  \\
\cdots  \ar[r] 
&  Q'_{t}   \ar[r]_-{\delta'_{t}} 
 &  Q'_{t-1}  \ar[r]
& \dots   \ar[r] 
& Q'_1  \ar[r]_-{\delta'_1}  \ar[r] 
 & Q'_0  \ar[r]_-{\delta'_0}  
& M' \ar[r] 
& 0 }
  \] is a commutative diagram in $A$-$\mod$. In particular, if $t=m$, using (b) we get that $\varphi_m^{s,m} = \gamma_m^{s,m}$ corresponds to a lifting in degree $m$ of $\delta'_0   \gamma^{s-m,0}_0=g_{s-m}$.

\end{remark} 

 \subsection{Liftings in Case B} 

Suppose that the minimal projective resolutions of the $A$-modules $M$ and $M'$
satisfy condition  $(\ast B)$, and these projective resolutions are defined as in Proposition \ref{resolucion B}. 
A representative of $\hat g \in  \Ext_{A_f} ^n(M, M')$ is given by a morphism in $\Hom_{A_f}( \hat Q_n, M')$, that is,  
$\hat g =(0,g,0)$ with $g \in \Hom_A(Q_n, M')$.  \\

The Lifting Theorem yields the existence of morphisms $\hat \gamma_{m}=  (\gamma_m, \beta_m, \gamma_m)$ such that  
\[ {\tiny  \xymatrix{ 
\cdots \ar[r] &  \hat Q_{n+m}   \ar[r]^-{\hat \delta_{n+m}} \ar[d]^{\hat \gamma_{m}} &  \hat Q_{n+m-1}    \ar[r] \ar[d]^{\hat \gamma_{m-1}} &
\dots  \ar[r] &   \hat Q_{n+1}  \ar[r]^{\hat \delta_{n+1}} \ar[d]^{\hat \gamma_{1}} &  \hat Q_{n} \ar[d]^{\hat \gamma_0} \ar[rd]^{\hat g} &  \\
\cdots  \ar[r] &  \hat Q'_m   \ar[r]^{\hat \delta'_{m}} &   \hat Q'_{m-1}    \ar[r] &  \dots   \ar[r] &  \hat Q'_1  \ar[r]^{\hat \delta'_1}  \ar[r] & \hat Q'_0  \ar[r]^{\hat \delta'_0}  & \hat M' \ar[r] & 0  } }\]
is a commutative diagram 
with $\gamma_m \in \Hom_A( Q_{n+m}, Q'_m)$ and $\beta_m \in \Hom_\K( Q_{n+m}, Q'_m)$.

\begin{lemma}\label{b--g-1} 
If $n$ is even and $\{ \hat \gamma_{m}= (\gamma_m, \beta_m, \gamma_m)\}_{m \geq 0}$ is a lifting for $\hat g$ then $\{\gamma_m\}_{m \geq 0}$ is a lifting for $g$.
\end{lemma}

\begin{proof}
If $n$ is even then $m$ and $n+m$ have the same parity. So, $\hat g=\hat  \delta'_0 \hat \gamma_0$ implies that $g= \delta'_0 \gamma_0$ and, for any $m \geq 1$,  $\hat \gamma_{m-1} \hat \delta_{n+m} = \hat \delta'_m \hat \gamma_m$ implies that 
\begin{align*}
(\gamma_{m-1}, \beta_{m-1}, \gamma_{m-1}) (\delta_{n+m}, T_{\delta_{n+m}}, \delta_{n+m}) & = (\delta'_{m}, T_{\delta'_{m}}, \delta'_{m})  (\gamma_m, \beta_m, \gamma_m) & \mbox{if $m$ is odd,} \\
(\gamma_{m-1}, \beta_{m-1}, \gamma_{m-1}) (0, \delta_{n+m},0) & = (0, \delta'_{m}, 0)  (\gamma_m, \beta_m, \gamma_m) & \mbox{if $m$ is even.} 
\end{align*}
In both cases we get that $\gamma_{m-1} \delta_{n+m}= \delta'_m \gamma_m$ for any $m \geq 1$.
 \end{proof}
 
 \medskip

\begin{lemma}\label{b--g-2} 
If $n$ is odd and $\{\gamma_m\}_{m \geq 0}$ is a lifting for $g$ then there exists a lifting  $\{\hat \gamma_m\}_{m \geq 0}$ for $\hat g$ defined by  
\[ \hat \gamma_{m} = 
\begin{cases}
(\gamma_m, T_{\gamma_m}, \gamma_m) & \mbox{if $m$ is even} \\
(0, \gamma_m, 0) & \mbox{if $m$ is odd} 
\end{cases}\]
where  $[T_{\gamma_m}(x)] = \tilde{f}([x] \otimes D_{m+n})E_{m}$, 
 with $D_{m+n}\in M_{n_{m+n} \times n_{m}}(A)$ the matrix associated to the morphism $\gamma_m : Q_{m+n} \to Q_{m}$ with $Q_i= \oplus_{r=1}^{n_i} A e_{i_r}$ and  $E_{m} \in M_{n_{m}}(A)$ the diagonal matrix associated to the projective module $Q_{m}$.
\end{lemma}

\begin{proof}
If $n$ is odd, then $\hat  \delta'_0 \hat \gamma_0 = (0, \delta'_0, 0)(\gamma_0, T_{\gamma_0}, \gamma_0)  = (0,  \delta'_0 \gamma_0, 0) = \hat g$. If $m$ is odd, then
\begin{align*}
\hat \gamma_{m-1} \hat \delta_{n+m} &  = (\gamma_{m-1}, T_{\gamma_{m-1}}, \gamma_{m-1}) (0, \delta_{n+m}, 0) \\
& = (0, \gamma_{m-1} \delta_{n+m}, 0) = (0,  \delta'_m \gamma_m, 0) \\
& = (\delta'_{m}, T_{\delta'_{m}}, \delta'_{m}) (0, \gamma_m, 0) =  \hat \delta'_m \hat \gamma_m.
\end{align*}
Finally, if $m$ is even, then
\begin{align*}
\hat \gamma_{m-1} \hat \delta_{n+m} &  = (0, \gamma_{m-1}, 0) ( \delta_{n+m},T_{\delta_{n+m}}, \delta_{n+m}) \\
& = (0, \gamma_{m-1} \delta_{n+m}, 0) = (0,  \delta'_m \gamma_m, 0) \\
& = (0, \delta'_{m}, 0) (\gamma_m, T_{\gamma_m} , \gamma_m) =  \hat \delta'_m \hat \gamma_m
\end{align*}
and the lemma follows.
\end{proof}

\section{Ext-algebras}

In the remainder of the paper we assume that all the simple $A$-modules satisfy condition $(\ast A)$ or  $(\ast B)$.

Set $S$ the direct sum of all the simple $A$- modules. We have seen in Theorem \ref{ext} that $\Ext^*_{A_f}(S, S) \simeq \Ext^*_A(S, S) \otimes_\K \K[x]$ in case (A) and 
$\Ext^*_{A_f}(S, S) \simeq \Ext^*_A(S, S)$ in case (B).
More precisely,  since  $S$ is semisimple, if we apply the functor 
$\Hom_{A_f}(-, S)$ to a minimal projective resolution for $S$ as described in Theorem \ref{3.4} and  \ref{resolucion B},  the corresponding morphisms vanish, and hence
	$$\Ext^n_{A_f}(S, S) =  \Hom_{A_f} (\bigoplus_{k=0}^n \hat Q_k, S) \simeq   \bigoplus_{k=0}^n \Hom_A(Q_k, S) = \bigoplus_{k=0}^n \Ext^k_A(S, S)$$
	in case (A), and 
	$$\Ext^n_{A_f}(S, S) =  \Hom_{A_f} ( \hat Q_n, S) \simeq  \Hom_A(Q_n, S) =  \Ext^n_A(S, S)$$
	in case (B).
	
	Now we will see that, using this isomorphism, the Yoneda product
	$$\circ \colon \Ext^m_{A_f}(S, S)  \otimes \Ext^n_{A_f}(S, S) \to  \Ext^{n+m}_{A_f}(S, S)$$ 
	can be described in terms of the Yoneda products
	$$\star \colon \Ext^t_{A}(S, S)  \otimes   \Ext^s_{A}( S, S) \to  \Ext^{s+t}_{A}(S,S)$$ 
	for all $s,t$ with $0 \leq s \leq n, 0\leq t \leq m$. 
	
 \subsection{The product in case (A)}

 In this case the  $A$-module $S$ 
satisfies condition $(\ast A)$ and its projective resolution is defined as in Proposition \ref{resolucion star}. \\

This case can be applied for any algebra $A$ with a Hochschild $2$-cocycle $f$ satisfying $\Im f \subset \rad A$ since, in this case, $\Im T_{\delta_1} \subset \rad Q_0$ and hence $\delta_0 T_{\delta_1}=0$,  see for instance Examples \ref{example 1} and \ref{example 2}.

\begin{lemma} \label{lifting} If  $S$ satisfies $(\ast A)$,
 let $\hat g \in  \Ext^n_{A_f}(S, S)$ and $\varphi_0^{s,m}$ as defined in \eqref{eq:varphi}. 
Let $0 \leq t \leq m$ and $0 \leq s \leq n+m$. Then, for $s<m+n$, the morphism 
$\varphi_0^{s,m}$ corresponds to a lifting in  degree $0$ of
\[\begin{cases}
 \sum_{k=0}^{r} \binom{r}{k} (-1)^{k}g_{s-2k} T_{\delta_{s-2k+1} }\cdots T_{\delta_{s}}, & \mbox{ if } m = 2r, 
 \\
 \sum_{k=0}^{r}\  \binom{r}{k} (-1)^{k+s}	( g_{s-2k} T_{\delta_{s-2k +1}} \cdots T_{\delta_{s}}- g_{s-2k-1} T_{\delta_{s-2k}} \cdots T_{\delta_{s}}), & \mbox{ if } m = 2r+1, 
\end{cases}\]
with the convention that $g_k=0$ if $k<0$ or $k>n$; and, for $s=m+n$, the morphism  $\varphi_0^{s,m}$ corresponds to a lifting in  degree $0$ of
$$(-1)^{m(n+1) + \frac{m(m+1)}{2}}g_n T_{\delta_{n+1}}T_{\delta_{n+2} }\cdots T_{\delta_{n+m}}.$$ 
\end{lemma}

\begin{proof} As we mentioned in Remark \ref{levantados}, the morphism  $\varphi_0^{s,m}$ corresponds to a lifting in degree $0$ of $\delta_0 \varphi_0^{s,m}$. If $s=0$, then  $\varphi_0^{0,m}= \gamma_0^{0,0}$ and $\delta_0  \gamma_0^{0,0} = g_0$, and hence the formula holds for any $m$. Assume $s>0$. 
Now we proceed by induction on $m$. For $m=0$, we have that $\delta_0 \varphi_0^{s,0}=g_s$.
If $s< n+m$, (f) implies
\[ - \varphi_{0}^{s,m} 
 + (-1)^s \varphi_{0}^{s,m-1}  =  \Phi_1^{s,m},\]
for
$ \Phi_{1}^{s,m} =  \delta_1  \phi_1^{s,m}  - \phi_{0}^{s-1, m-1}  \delta_{s} +  T_{\delta_1}   \varphi_1^{s,m} + (-1)^{s} \varphi_{0}^{s-1, m-1}  T_{\delta_{s}}$.
Notice that $\delta_0\delta_1=0$,  $\delta_0  \phi_{0}^{s-1, m-1}  \delta_{s}=0$ and  $ \delta_0  T_{\delta_1 } \varphi_1^{s,m}=0$ since $\delta_0 \colon Q_0 \to S$, $\Im T_{\delta_1} \subset \rad Q_0$ and $\Im \delta_s \subset \rad Q_{s-1}$. Hence,
 \[ \delta_0  \varphi_0^{s,m} =  (-1)^s  \delta_0 \varphi_{0}^{s,m-1} + (-1)^{s-1}  \delta_0 \varphi_{0}^{s-1, m-1}  T_{\delta_{s}}.\]
 Now, by inductive hypothesis, we have that 
   \begin{align*} 
  \delta_0  \varphi_0^{s,m} =  \sum_{k=0}^{r}\  \binom{r}{k} 	(-1)^{k+s} (g_{s-2k} T_{\delta_{s-2k +1}} \cdots T_{\delta_{s}}- g_{s-2k-1} T_{\delta_{s-2k}} \cdots T_{\delta_{s}}),
  \end{align*}
  for  $m = 2r +1$,
  and   for $m = 2r+2$, we have
   \begin{align*}
    \delta_0  \varphi_0^{s,m}  
  =&  \sum_{k=0}^{r}\  \binom{r}{k} (-1)^{k}	( g_{s-2k} T_{\delta_{s-2k +1}} \cdots T_{\delta_{s}}- g_{s-2k-1} T_{\delta_{s-2k}} \cdots T_{\delta_{s}} \\
 & +   g_{s-1-2k} T_{\delta_{s-1-2k +1}} \cdots T_{\delta_{s}}- g_{s-1-2k-1} T_{\delta_{s-1-2k}} \cdots T_{\delta_{s}}) \\ 
   =& g_{s} T_{\delta_{s+1}} \cdots T_{\delta_{s}}
   + \sum_{k=1}^{r}\  \binom{r}{k} 	(-1)^{k} g_{s-2k} T_{\delta_{s-2k +1}} \cdots T_{\delta_{s}} \\ 
& + \sum_{l=1}^{r} \binom{r}{k-1}  (-1)^{k} g_{s-2k} T_{\delta_{s-2k+1}} \cdots T_{\delta_{s}}  +   (-1)^{r+1} g_{s-2r-2} T_{\delta_{s-2r-1}} \cdots T_{\delta_{s}} \\
   =& \sum_{k=0}^{r+1} \binom{r+1}{k} (-1)^{k}g_{s-2k} T_{\delta_{s-2k+1}} \cdots T_{\delta{s}}
 \end{align*}
using that $\binom{r}{k} + \binom{r}{k-1} = \binom{r+1}{k}$.
Finally, arguing as above, if $s = n+m$ equation $(g)$  implies
\[\delta_0 \varphi_{0}^{n+m,m}  = (-1)^{n+m+1} \delta_0  \varphi_{0}^{n+m-1,m-1} T_{\delta_{n+m}}.\] 
Therefore, by  inductive hypothesis, we get 
\[\delta_0  \varphi_{0}^{n+m-1,m-1} T_{\delta_{n+m} }=(-1)^{(m-1)(n+1) + \frac{(m-1)m}{2}}g_n T_{\delta_{n+1}} \cdots T_{\delta_{n+m}}\]
and the lemma follows.
\end{proof}
	
The next theorem describes the algebra structure of  $\Ext^*_{A_f}(S, S)$ in terms of the Yoneda product in $\Ext^*_A(S,S)$.
As we mention in Theorem \ref{ext}, we use the isomorphism 
$$\Ext^*_{A_f}(S, S) \simeq \Ext^*_A(S, S) \otimes_\K \K[x],$$
 thus we identify $\hat g$ with $\sum_{i=0}^n g_i \ x^{n-i}$ for any $\hat g\in \Ext^n_{A_f}(S, S)$. 
Recall that we denote by $\circ$ and $\star$ the Yoneda products in $\Ext^*_{A_f}(S, S)$ and $\Ext^*_{A}(S, S)$, respectively.
\begin{theorem}\label{thm yoneda}
	If  $S$ satisfies $(\ast A)$ then 
	the product of $\hat g \in \Ext^n_{A_f}(S, S)$ 
			and $\hat h \in \Ext^m_{A_f}( S, S)$ is given by
\[ h_i x^{m-i} \  \circ  \ g_s \ x^{n-s} =    (-1)^{s(m-i)} h_i  \star g_s \ x^{m+n-i-s}  +  \sum_{l=s+1}^{m+n-i} c_{i,s,l} \ h_i  \star g_{s} T_{\delta_{s+1}} \cdots T_{\delta_{l}}  \ x^{m+n-i-l} 
\]
for certain integers $c_{i,s,l}$, and using linearity.
\end{theorem}

\begin{proof}
		Let $\hat g \in \Ext^n_{A_f}(S, S)$ 
		and $\hat h \in \Ext^m_{A_f}( S, S)$. 
 We use the matrix notation $[h_0 \dots h_m]$ for $\sum_{i=0}^m h_ix^{m-i}$.  From Proposition \ref{Proposition 6.1} we have that 
	$$\hat h \circ  \hat g  = \hat h \ \hat \varphi_m =  
	(0, h  \varphi_m , 0)$$ 
with
	\[
	 h   \varphi_m   = [h_0 \ \dots h_m] \begin{pmatrix}
		\varphi_0^{0,m} & \varphi_0^{1,m} & \varphi_0^{2,m} &  \dots & \varphi_0^{m,m} &    \dots &  \varphi_0^{n+m,m} \\
		0 & \varphi_1^{1,m} & \varphi_1^{2,m} & \dots &    \varphi_1^{m,m}    & \dots &  \varphi_1^{n+m,m} \\ 
		\vdots & &  &  & \vdots &  &  \vdots \\
		0 &   \dots &    & 0 & \varphi_{m}^{m,m} &  \dots   &  \varphi_{m}^{n+m,m} 
	\end{pmatrix}. \]
	Hence
	\[ \hat h \circ \hat g = ( \sum_{i=0}^m  h_i \ x^{m-i}) \circ ( \sum_{i=0}^n g_i \ x^{n-i}) = \sum_{i=0}^m \sum_{j=i}^{n+m} h_i  \varphi^{j,m}_i \ x^{m+n-j}. \]
	As mentioned in Remark \ref{levantados},  $\varphi^{s,m}_t$  corresponds to a lifting in degree $t$  of $ \delta_0 \varphi^{s-t, m-t}_0$
	and, by the definition of the Yoneda product, we have that 
	\begin{align*}
	 \hat h \circ \hat g & =  \sum_{i=0}^m \sum_{j=i}^{n+m} h_i  \star \delta_0 \varphi^{j-i,m-i}_0 \ x^{m+n-j} =  \sum_{i=0}^m \sum_{s=0}^{n+m-i} h_i  \star \delta_0 \varphi^{s,m-i}_0 \ x^{m+n-i-s}.
	 \end{align*}
	By Lemma \ref{lifting},  if $s<n+m$, $ h_i \star  \delta_0 \varphi^{s,m-i}_0$ equals to 
		\[ \begin{cases}
	\sum_{k=0}^{r} \binom{r}{k} (-1)^{k} h_i \star g_{s-2k} T_{\delta_{s-2k+1}} \cdots T_{\delta_{s}}
	\\
	 \sum_{k=0}^{r}\  \binom{r}{k} 	(-1)^{k+s} (h_i \star g_{s-2k} T_{\delta_{s-2k +1}} \cdots T_{\delta_{s}} - h_i \star  g_{s-2k-1} T_{\delta_{s-2k}} \cdots T_{\delta_{s}}) 	 \end{cases} \]
when $m-i$ equals $2r$ or $2r+1$ respectively.  Also, if $s=n+m$,
	  then $i=0$ and 
$$ h_0 \star  \delta_0 \varphi^{s,m-i}_0= (-1)^{m(n+1) + \frac{m(m+1)}{2}} h_0 \star g_n T_{\delta_{n+1}} T_{\delta_{n+2}} \cdots T_{\delta_{n+m}}.$$
Hence,
\[ \hat h \circ \hat g  =  \sum_{i=0}^m \sum_{s=0}^{n} (-1)^{s(m-i)} h_i  \star g_s \ x^{m+n-i-s}  +  \sum_{i=0}^m \sum_{s=1}^{n+m-i} \sum_{l=0}^{s-1} c_{i,s,l} \ h_i  \star g_{l} T_{\delta_{l+1}} \cdots T_{\delta_{s}}  \ x^{m+n-i-s}  
\]
with the convention that $g_l=0$ if  $l>n$.	
\end{proof}

\begin{corollary}\label{previous}
 If   $\Im T_{\delta_i} \subset \rad Q_{i-1}$ for all $i >0$ then $S$ satisfies $(\ast A)$ and
$$\Ext^*_{A_f}(S, S) \simeq \Ext^*_A (S,S) \otimes \K[x]$$
as graded algebras. That is,  
		\[ \hat h \circ \hat g  =  \sum_{i=0}^m \sum_{s=0}^{n} (-1)^{s(m-i)} h_i  \star g_s \ x^{m+n-i-s} \]
 for all $\hat g = \sum_{i=0}^n g_i \ x^{n-i}\in \Ext^n_{A_f}(S, S)$ and $ \hat h =  \sum_{i=0}^m  h_i \ x^{m-i}\in\Ext^m_{A_f}(S, S)$.
 \end{corollary}
 
  \subsection{The product in case (B)}
  
In this case the  $A$-module $S$ 
satisfies condition $(\ast B)$ and its projective resolution is defined as in Proposition \ref{resolucion B}. 

Notice that all the algebras $A$ considered in this case satisfy the following property: for each vertex $i$ the simple module $S_i$ is a direct summand of the $\K$-vector space $\Im f$. This remark follows from condition \ref{itemtwo}  and the fact that $\Im \delta_{1} = \rad Q_{0} = \rad A e_i$, and it implies that the quiver $Q$ has oriented cycles going through $i$, for all vertices $i$.

In particular all the simple modules in Examples \ref{example 3.5} and \ref{example 3.6} satisfy condition $(\ast B)$. In fact, for each vertex $i$ there exist a nonzero path $w \in e_i(\rad A)e_j$ and an arrow $e_j \alpha e_i$ such that $w \alpha=0= \alpha w$ and $f(w \otimes \alpha)=e_i$. The simple module $S_i$ admits a minimal projective resolution 
\[ \cdots   \to Ae_j \stackrel{\delta_1}  {\longrightarrow}  Ae_i  \stackrel{\delta_2} {\longrightarrow}   Ae_j \stackrel{\delta_1}  {\longrightarrow}  Ae_i  \stackrel{\delta_0}  {\longrightarrow} S_i\to 0 \]
such that $\delta_{2}(x) = x w, \delta_{1}(x) = x \alpha$ and $\ker \delta_{1} = \K w$.
Now,  since
\begin{align*}
\delta_{0}  T_{\delta_{1}} (w) &=  \delta_0 (f(w \otimes \alpha)) =   \delta_0(e_i) = e_i, \\
\delta_{2}  T_{\delta_{1}} (w) &=  T_{\delta_{1}} (w) w = f(w \otimes \alpha) w = w,
\end{align*}
condition \ref{itemone} follows. Condition \ref{itemtwo} can be deduced from the equality
\[ T_{\delta_{1}}(w) = f(w \otimes \alpha)=e_i\]
and the fact that $\Im \delta_1= \rad A e_i$.

Now we describe the algebra structure of  $\Ext^*_{A_f}(S, S)$ in terms of the Yoneda product in $\Ext^*_A(S,S)$ when 
$$\Ext^*_{A_f}(S, S) \simeq \Ext^*_A(S, S)$$
denoting by $\circ$ and $\star$ the Yoneda products in $\Ext^*_{A_f}(S, S)$ and $\Ext^*_{A}(S, S)$, respectively.
\begin{theorem}\label{thm yoneda b}
	If  $S$ satisfies $(\ast B)$ then 
	the product of $\hat g \in \Ext^n_{A_f}(S, S)$ 
			and $\hat h \in \Ext^m_{A_f}( S, S)$ is given by
	\[ 
	 \hat  h  \circ \hat g = \begin{cases}
	 0 &\mbox{if $m$ and $n$ are odd,}  \\
	  h \star g &\mbox{in all the other cases.} \\
	 \end{cases}\]	
	\end{theorem}

	\begin{proof}
		Let $\hat g \in \Ext^n_{A_f}(S, S)$ 
		and $\hat h \in \Ext^m_{A_f}( S, S)$. 
  From Lemma \ref{b--g-1} we have that, if $n$ is even then
  \[ \hat h \circ  \hat g  =  (0, h, 0) (\gamma_m, \beta_m, \gamma_m) = (0, h \gamma_m, 0)= h \star g \]
  since $\{\gamma_m\}_{m \geq 0}$ is a lifting for $g$. Now, if $n$ is odd we use Lemma \ref{b--g-2} and we get that for
 $m$  even we have
   \[ \hat h \circ  \hat g  =  (0, h, 0) (\gamma_m, \beta_m, \gamma_m) = (0, h \gamma_m, 0)= h \star g \]
  and for $m$  odd, we have 
   $ \hat h \circ  \hat g  =  (0, h, 0) (0, \gamma_m, 0) = 0$.	
	\end{proof}

\begin{example}
Let $A$ and $f$  be as in Example \ref{example 1}, $S$ satisfies condition $(\ast A)$. In this case, a direct computation shows that $\Im T_{ \delta_i } \subset \rad Q_{i-1}$ for any $i >0$. By Corollary \ref{previous} we have that $\Ext^*_{A_f}(S, S) \simeq \Ext^*_A (S,S) \otimes \K[x]$ with the product given by 
$$\hat h \circ \hat g  =  \sum_{i=0}^m \sum_{s=0}^{n} (-1)^{s(m-i)} h_i  \star g_s \ x^{m+n-i-s}.$$
 \end{example}

\begin{example}
 For $A$ and $f$ as in Example \ref{example 2}, $S$ satisfies condition $(\ast A)$. One can check that $\Im T_{\delta_{2k-1}} \subset \rad Q_{2k-2}$ and $\Im T_{\delta_{2k}} \not \subset \rad Q_{2k-1}$, for any $k > 0$. 
 More precisely,  $T_{\delta_{2k-1}} : P_1 \oplus P_2 \to P_2 \oplus P_1$ and $T_{\delta_{2k}} : P_1 \oplus P_2 \to P_1 \oplus P_2$ are defined by
\[  [T_{\delta_{2k-1}}(x_1, x_2)] =\tilde{f}  \left (
\begin{bmatrix}
	x_1 &
	x_2
\end{bmatrix} \otimes 
\begin{bmatrix}
	\alpha_1 & 0 \\
	0 & \alpha_2
\end{bmatrix}  \right )
\begin{bmatrix}
	e_2 & 0\\
	0 & e_1
\end{bmatrix},\]
\[  [T_{\delta_{2k}}(x_1, x_2)] = -\tilde{f}  \left (
\begin{bmatrix}
	x_1 &
	x_2
\end{bmatrix} \otimes 
\begin{bmatrix}
	\alpha_1 \alpha_2 & 0 \\
	0 & \alpha_2 \alpha_1
\end{bmatrix}  \right )
\begin{bmatrix}
	e_1 & 0 \\
	0 & e_2
\end{bmatrix} +
\begin{bmatrix}
	x_1 &
	x_2
\end{bmatrix}
\begin{bmatrix}
	e_1 & 0 \\
	0 & e_2
\end{bmatrix}.\]
  Then, by Theorem \ref{thm yoneda}, the Yoneda product in $\Ext_{A_f}(S,S)$ is given by  
\[ \hat h \circ \hat g = \sum_{i=0}^m \sum_{s=0}^{n}  (-1)^{s(m-i)} h_i  \star g_s \ x^{m+n-i-s}   + \sum_{i=0}^m \sum_{k=1}^{\frac{n+m-i}{2}}  c_{i,2k,2k-1} \ h_i  \star g_{2k-1}    \ x^{m+n-i-2k} \]
since  $g_{2k-1}  \alpha_{2k} =  g_{2k-1}$  in this particular example. For instance, when $n = m = 1$, we obtain $$\hat h\circ \hat g=  h_0\star g_0 \ x^2+ (h_1\star g_0 - h_0\star g_1)x + (h_1\star g_1 - h_0\star g_1).$$
\end{example}

\begin{example} The description of the Ext-algebra of $A= \K [\alpha]/ \langle \alpha^r \rangle$  as graded algebra is well known:
\[
\Ext^*_{A}(S, S) \simeq 
\begin{cases}
\K[x] &\mbox{ if $r=2$,} \\
\K[u,v]/ \langle u^2 \rangle &\mbox{ if $r>2$, with $u$ of degree $1$ and $v$ of degree $2$.}
\end{cases}
\]
For $f$ as in Example \ref{example 3.5}, we have that $A_f= \K [\alpha]/\langle \alpha^{2r} \rangle$.  Then it is clear that $\Ext^*_{A_f}(S, S) \simeq 
\Ext^*_{A}(S, S)$ as graded algebras for every $r>2$.
\end{example}

\end{document}